\documentclass[11pt,reqno]{amsart}

\usepackage[ansinew]{inputenc}
\usepackage{textcomp}
\usepackage{cite}

\usepackage{amsthm}
\usepackage{amssymb}
\usepackage{amsfonts}
\usepackage{amsmath}

\usepackage{graphicx}

\newtheorem{theorem}{Theorem}

\newtheorem{lemma}[theorem]{Lemma}
\newtheorem{corollary}[theorem]{Corollary}

\theoremstyle{remark}
\newtheorem*{remark}{Remark}

\numberwithin{equation}{section}

\newcommand{\lk}{\left(}
\newcommand{\rk}{\right)}
\newcommand{\R}{\mathbb{R}}
\newcommand{\La}{\Lambda}
\newcommand{\tr}{\textnormal{Tr}}

\title{Geometrical Versions of improved Berezin-Li-Yau Inequalities}

\author{Leander Geisinger}
\address{Leander Geisinger \\ Universität Stuttgart \\ Pfaffenwaldring 57 \\ D - 70569 Stuttgart}
\email{geisinger@mathematik.uni-stuttgart.de}

\author{Ari Laptev}
\address{Ari Laptev \\ Imperial College London \\ 180 Queen's Gate \\ London SW7 2AZ \\ UK }
\email{a.laptev@imperial.ac.uk}

\author{Timo Weidl} 
\address{Timo Weidl \\Universität Stuttgart \\ Pfaffenwaldring 57 \\ D - 70569 Stuttgart}
\email{weidl@mathematik.uni-stuttgart.de}

\subjclass[2000]{Primary 35P15; Secondary 47A75}

\keywords{}

\begin{document}

\begin{abstract}
We study the eigenvalues of the Dirichlet Laplace operator on an arbitrary bounded, open set in $\R^d$, $d \geq 2$. In particular, we derive upper bounds on Riesz means of order $\sigma \geq 3/2$, that improve the sharp Berezin inequality by a negative second term. This remainder term depends on geometric properties of the boundary of the set and reflects the correct order of growth in the semi-classical limit.

Under certain geometric conditions these results imply new lower bounds on individual eigenvalues, which improve the Li-Yau inequality.
\end{abstract}

\maketitle

%%%%%%%%%%%%%%%%%%%%%%%%%%%%%%%%%%%%%%%%%%%%%%%%%%%%%%%%%%%%%%%%

\section{Introduction}
\label{sec:intro}

Let $\Omega \subset \R^d$ be an open set and let $-\Delta$ denote the Dirichlet Laplace operator on $L^2(\Omega)$, defined as a self-adjoint operator with form domain $H^1_0(\Omega)$. We assume that the volume of $\Omega$, denoted by $|\Omega|$, is finite. Then the embedding $H_0^1(\Omega) \hookrightarrow L^2(\Omega)$ is compact and the spectrum of $-\Delta$ is discrete: It consists of positive eigenvalues 
$$
0\, < \, \lambda_1 \, \leq \, \lambda_2 \, \leq \, \lambda_3 \, \leq \, \dots
$$
accumulating at infinity only.

Here we are interested in upper bounds on the Riesz means 
$$
\sum_k(\La - \lambda_k)_+^\sigma \, = \, \tr \lk -\Delta - \La \rk_-^\sigma \, , \quad \sigma \geq 0 \, ,
$$
where we use the notation $x_\pm = (|x|\pm x) /2$.
In 1972 Berezin proved that, convex eigenvalue means are bounded uniformly by the corresponding phase-space volume, see \cite{Berezi72}: For any open set $\Omega \subset \R^d$, $\sigma \geq 1$, and all $\La >0$ 
\begin{equation}
\label{in:beliyau}
\tr \lk - \Delta - \Lambda \rk_-^\sigma \, \leq \, \frac{1}{(2\pi)^d} \iint_{\Omega \times \R^d} \lk |p|^2 - \La \rk^\sigma_+ dp \, dx \, =  \, L^{cl}_{\sigma,d} \, |\Omega| \, \La^{\sigma+d/2}  \, ,
\end{equation}
see also \cite{LiYau83}, where the problem is treated from a different point of view. Here $L^{cl}_{\sigma,d}$ denotes the so-called Lieb-Thirring constant
$$
L^{cl}_{\sigma,d} \, = \, \frac{\Gamma(\sigma+1)}{(4 \pi)^{d/2} \, \Gamma(\sigma+1+ d/2)} \, . 
$$

The Berezin inequality (\ref{in:beliyau}) captures, in particular, the well-known asymptotic limit that goes back to Hermann Weyl \cite{Weyl12}: For $\Omega \subset \R^d$ and $\sigma \geq 0$ the asymptotic identity
\begin{equation}
\label{eq:as_basic}
\tr (-\Delta-\Lambda)_-^\sigma \, = \, L^{cl}_{\sigma,d} \, |\Omega| \, \La^{\sigma + d/2} + o \lk \La^{\sigma+d/2} \rk 
\end{equation}
holds true as $\La \to \infty$. From this follows, that the Berezin inequality is sharp, in the sense that the constant in (\ref{in:beliyau}) cannot be improved. However, Hermann Weyl's work stimulated further analysis of the asymptotic formula and (\ref{eq:as_basic}) was gradually improved by studying the second term, see \cite{CouHil24,Hoerma68,Ivrii80,Melros80,SafVas97,Ivrii98} and references within. The precise second term was found by Ivrii \cite{Ivrii80}: Under appropriate conditions on the set $\Omega$ and its boundary $\partial \Omega$ the relation
\begin{equation}
\label{eq:asymptotics}
\tr (-\Delta-\Lambda)_-^\sigma \, = \, L^{cl}_{\sigma,d} \, |\Omega| \, \La^{\sigma + d/2} - \frac 14 \, L^{cl}_{\sigma,d-1} \, |\partial \Omega| \, \La^{\sigma+(d-1)/2} + o \lk \La^{\sigma+(d-1)/2} \rk 
\end{equation}
holds as $\La \to \infty$. To simplify notation we write $|\Omega|$ for the volume (the $d$-dimensional Lebesgue measure) of $\Omega$, as well as $|\partial \Omega|$ for the $d-1$-dimensional surface area of its boundary. Since the second term of this semi-classical limit is negative, the question arises, whether the Berezin inequality (\ref{in:beliyau}) can be improved by a negative remainder term.

Recently, several such improvements have been found, initially for the discrete Laplace operator, see \cite{FrLiUe02}. The first result for the continuous Laplace operator is due to Melas \cite{Melas03}. From his work follows that 
\begin{equation}
\label{eq:melas}
\tr \lk -\Delta - \La \rk^\sigma_- \, \leq \, L^{cl}_{\sigma,d} \, |\Omega|   \lk  \La - M_d \frac{|\Omega|}{I(\Omega)} \rk_+^{\sigma+d/2} \, , \quad \La > 0 \, , \quad \sigma \geq 1 \, ,
\end{equation}
where $M_d$ is a constant depending only on the dimension and $I(\Omega)$ denotes the second moment of $\Omega$,
see also \cite{Ilyin09,Yolcu09} for further generalisations. One should mention, however, that these corrections do not capture the correct order in $\La$ from the second term of the asymptotics (\ref{eq:asymptotics}). This was improved in the two-dimensional case in \cite{KoVuWe08}, where it is shown that one can choose the order of the correction term arbitrarily close to the correct one.

In this paper we are interested in the case $\sigma \geq 3/2$. 
For these values of $\sigma$ it is known, \cite{Weidl08}, that one can strengthen the Berezin inequality for any open set $\Omega \subset \R^d$ with a negative remainder term reflecting the correct order in $\La$ in comparison to the second term of (\ref{eq:asymptotics}). However, since one can increase $|\partial \Omega|$ without changing the individual eigenvalues $\lambda_k$ significantly, a direct analog of the first two terms of the asymptotics (\ref{eq:asymptotics}) cannot yield a uniform bound on the eigenvalue means. Therefore - without further conditions on $\Omega$ -  any uniform improvement of (\ref{in:beliyau}) must invoke other geometric quantities. 

In the result from \cite{Weidl08} the remainder term involves certain projections on $d-1$-dimensional hyperplanes. In \cite{GeiWei10} a universal improvement of (\ref{in:beliyau}) was found, that holds for $\sigma \geq 3/2$ with a correction term of correct order, depending only on the volume of $\Omega$. 

The proof of the aforementioned results relies on operator-valued Lieb-Thirring inequalities \cite{LapWei00} and an inductive argument, that allows to reduce the problem to estimating traces of the one-dimensional Dirichlet Laplace operator on open intervals.

In this paper we use the same approach, but with new estimates in the one-dimensional case, in order to make the dependence on  the geometry more transparent. 
The new one-dimensional bounds involve the distance to the boundary of the interval in question and are related to Hardy-Lieb-Thirring inequalities for Schr\"odinger operators, see \cite{EkhFra06}. There it is shown that for $\sigma \geq 1/2$ and potentials $V \in L^{\sigma + 1/2}(\R_+)$, given on the half-line $\R_+ = (0,\infty)$, the inequality
$$
\tr \lk - \frac{d^2}{dt^2} - V \rk_-^\sigma \, \leq \, L_{\sigma} \int_0^\infty \lk V(t)-\frac{1}{4t^2} \rk_+^{\sigma+1/2} dt
$$
holds true, with a constant $L_\sigma$ independent of $V$. For further developments see \cite{FrLiSe08,Frank09}.

We start this paper with analysing the special case of the Dirichlet Laplace operator given on a finite interval $I \subset \R$, with the constant potential $V \equiv \La$. For $\sigma \geq 1$ we establish that the estimate
$$
\tr \lk - \frac{d^2}{dt^2} - \La \rk_-^\sigma \, \leq \, L^{cl}_{\sigma,1} \int_I \lk \La-\frac{1}{4\delta(t)^2} \rk_+^{\sigma+1/2} dt
$$
is valid with the sharp constant $L^{cl}_{\sigma,1}$, where $\delta(t)$ denotes the distance to the boundary of $I$. This is done in section \ref{sec:1d}.

Then we can use results from \cite{LapWei00,Weidl08} to deduce bounds in higher dimensions: In section \ref{sec:higherd} we first derive improvements of $(\ref{in:beliyau})$, which are valid for any open set $\Omega \subset \R^d$, $d \geq 2$. These improvements depend on the geometry of $\Omega$. In view of the asymptotic result (\ref{eq:asymptotics}) one might expect, that this geometric dependence can be expressed in terms of the boundary of $\Omega$. To see this, we adapt methods, which were used in \cite{Davies95,Davies99,HoHoLa02} to derive geometric versions of Hardy's inequality. Here the result gives an improved Berezin inequality with a correction term of correct order depending on geometric properties of the boundary.

If $\Omega$ is convex and smooth this dependence can be expressed only in terms of $|\Omega|$, $|\partial \Omega|$ and the curvature of the boundary. In particular the first remainder term of the estimate is very similar to the second term of the semiclassical asymptotics (\ref{eq:asymptotics}): it shows the same order in $\La$ and it depends only on the surface area of the boundary. 

In section \ref{sec:lower} we return to the general case, where $\Omega \subset \R^d$ is not necessarily convex or smooth, and obtain lower bounds on individual eigenvalues $\lambda_k$. Under certain conditions on the geometry of $\Omega$ these results improve the estimate
\begin{equation}
\label{in:liyau}
\lambda_k \, \geq \, C_d \, \frac{d}{d+2} \lk \frac{k}{|\Omega|} \rk^{2/d} \, ,
\end{equation}
from \cite{LiYau83}, where $C_d$ denotes the semi-classical constant $4 \pi \, \Gamma(d/2+1)^{2/d}$.

Finally, in section \ref{sec:2d}, we specialise to the two-dimensional case, where we can use the foregoing results and refined geometric considerations to further improve and generalise the inequalities. In particular, we avoid dependence on curvature, thus we do not require smoothness of the boundary.

The question whether such improved estimates can be generalised to $1 \leq \sigma < 3/2$ remains open.

%%%%%%%%%%%%%%%%%%%%%%%%%%%%%%%%%%%%%%%%%%%%%%%%%%%%%%%%%%%%%%%%%%%%%%%%%%%%%%

\section{One-dimensional considerations}
\label{sec:1d}

Let us consider an open interval $I \subset \R$  of length $l > 0$. For $t \in I$  let
$$
\delta(t) \, = \, \inf \left\{ |t-s| \, : \, s \notin I \right\}
$$
be the distance to the boundary. The eigenvalues of $-d^2/dt^2$ subject to Dirichlet boundary conditions at the endpoints of $I$ are given by $\lambda_k \, = \, k^2 \pi^2 / l^2$. Therefore the Riesz means equal
$$
\tr \lk - \frac{d^2}{dt^2} - \La \rk^\sigma_- \, = \, \sum_k \lk \La - \frac{k^2 \pi^2}{l^2} \rk_+^\sigma \, .
$$
To find precise bounds on the Riesz means in the one-dimensional case, it suffices to analyse this sum explicitly. Our main observation is

\begin{lemma}
\label{lem:1d}
Let $I \subset \R$ be an open interval and $\sigma \geq 1$. Then the estimate 
$$
\tr \lk - \frac{d^2}{dt^2} - \La \rk_-^\sigma \, \leq \, L^{cl}_{\sigma,1} \int_I \lk \La - \frac{1}{4\delta(t)^2}  \rk_+^{\sigma+1/2} dt \, ,
$$
holds true for all $\La > 0$. The constant $1/4$ cannot be improved.
\end{lemma}

The remainder of this section deals with the proof of this estimate. First we need two rather technical results, whose proof is elementary but not trivial and therefore will be given in the appendix.

\begin{lemma}
\label{lem:elementary}
For all $A \geq 1/\pi$
\begin{equation}
\label{in:elementary}
\sum_k \lk 1- \frac{k^2}{A^2} \rk_+ \, \leq \, \frac{2}{3 \pi} \int_1^{\pi  A} \lk 1 - \frac{1}{s^2} \rk^{3/2} \, ds \, .
\end{equation}
\end{lemma}

\begin{lemma}
\label{lem:asympt}
Let $I \subset \R$ be an open interval of length $l>0$. Then for $\sigma \geq 1$ and $c > 0$
$$
L^{cl}_{\sigma,1} \int_I \lk \La - \frac{c}{\delta(t)^2} \rk_+^{\sigma+1/2} dt - \sum_k \lk \La - \frac{\pi^2 k^2}{l^2} \rk_+^\sigma  \, = \, \lk \frac 12 - \sqrt c \rk \La^\sigma + o \lk \La^\sigma \rk
$$
holds as $\La \to \infty$.
\end{lemma}

\begin{proof}[Proof of Lemma \ref{lem:1d}]
Note that one can always assume $I = (0,l)$, where $l > 0$ denotes the length of the interval.
First we deduce the estimate for $\sigma = 1$ from Lemma \ref{lem:elementary}. Assume $\La \geq l^{-2}$ and apply Lemma \ref{lem:elementary} with $A = l \sqrt{\La} / \pi$ to get
$$
\tr \lk - \frac{d^2}{dt^2} - \La \rk_- \, = \, \La \sum_k \lk 1 - \frac{\pi^2 k^2}{l^2 \La} \rk_+ \, \leq \, \La \, \frac{2}{3 \pi} \int_1^{l \sqrt \La} \lk 1 - \frac{1}{s^2} \rk^{3/2} \, ds \, .
$$
Substituting $s = 2 t \sqrt \La$, we find that
$$
\tr \lk - \frac{d^2}{dt^2} - \La \rk_- \, \leq \, \frac{4}{3 \pi} \int_{1/(2\sqrt\La)}^{l/2} \lk \La - \frac{1}{4t^2} \rk^{3/2} dt \, = \, 2 L^{cl}_{1,1} \int_0^{l/2} \lk \La - \frac{1}{4t^2} \rk_+^{3/2} dt
$$
holds for all $\La \geq l^{-2}$. Note that this inequality is trivially true for $0 < \La < l^{-2}$, since the left hand side is zero. Finally, we use the identities
$$
\int_0^{l/2} \lk \La - \frac{1}{4t^2} \rk_+^{3/2} dt \, = \, \int_{l/2}^l \lk \La - \frac{1}{4(l-t)^2} \rk_+^{3/2} dt \, = \, \frac 12 \int_0^{l} \lk \La - \frac{1}{4\delta(t)^2} \rk_+^{3/2} dt
$$
to finish the proof for $\sigma = 1$.

To deduce the claim for $\sigma > 1$ we can apply a method from \cite{AizLie78}. Writing
$$
\sum_k \lk \La - \lambda_k \rk_+^\sigma \, = \, \frac{1}{B(2,\sigma-1)} \int_0^\infty \tau^{\sigma-2} \sum_k \lk \La - \tau - \lambda_k \rk_+ \, d\tau \, ,
$$
we estimate
\begin{eqnarray*}
\tr \lk - \frac{d^2}{dt^2} - \La \rk_-^{\sigma} & \leq & \frac{1}{B(2,\sigma-1)} L^{cl}_{1,1} \int_I \int_0^\infty \tau^{\sigma-2} \lk \La - \frac{1}{4 \delta(t)^2} - \tau \rk_+^{3/2} \, d\tau \, dt \\
& = & L^{cl}_{1,1} \frac{B( 5/2,\sigma-1)}{B(2,\sigma-1)} \int_I \lk \La - \frac{1}{4 \delta(t)^2} \rk_+^{\sigma+1/2} \, dt
\end{eqnarray*} 
and the result follows from the identity $L^{cl}_{1,1} B(5/2,\sigma-1) = L^{cl}_{\sigma,1} B(2,\sigma-1)$.

The claim, that the constant $1/4$ cannot be improved, follows from Lemma \ref{lem:asympt}: For $c = 1/4$ the leading term of the asymptotics in Lemma \ref{lem:asympt} vanishes. For any constant $c > 1/4$ the leading term is negative, thus the estimate from Lemma \ref{lem:1d} must fail in this case, for large values of $\La$. 
\end{proof}

Figure \ref{fig:diff} illustrates the results of Lemma \ref{lem:1d} and Lemma \ref{lem:asympt} for $l = \pi^2$ and $\sigma = 1$ with the sharp constant $c = 1/4$: The function
$$
f( \La) \, = \, L^{cl}_{1,1} \int_0^\pi \lk \La - \frac{1}{4 \, \delta(t)^2} \rk_+^{3/2} dt - \sum_k \lk \La - k^2 \rk_+
$$
is plotted for $1 < \La < 112$, so that the first ten minima are shown.

\begin{figure}[ht]
\centering
\includegraphics[width=10cm]{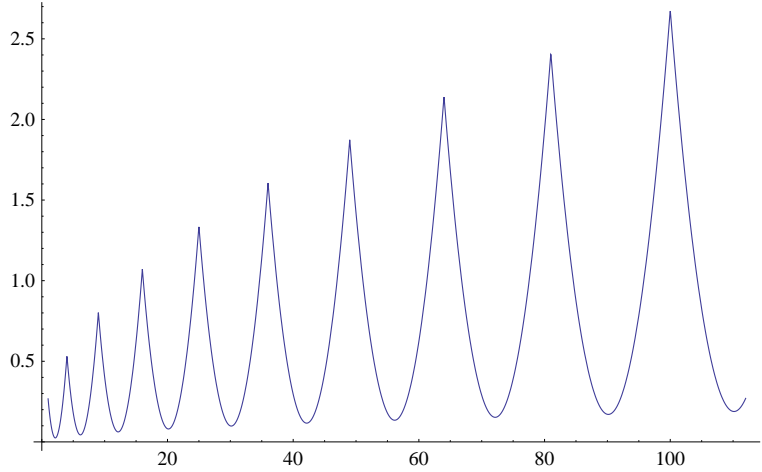}     %{maple2}
\caption{The function $f(\La)$, illustrating the results from section 2.}
\label{fig:diff}
\end{figure}

%%%%%%%%%%%%%%%%%%%%%%%%%%%%%%%%%%%%%%%%%%%%%%%%%%%%%%%%%%%%%%%%%%%%%%%%%

\section{Results in higher dimensions}
\label{sec:higherd}

In this section we use the one-dimensional result to prove uniform eigenvalue estimates for the Dirichlet Laplace operator in bounded open sets in higher dimensions. These estimates - refinements of the Berezin inequality (\ref{in:beliyau}) - depend on the geometry of the set, in particular on properties of the boundary.

\subsection{Arbitrary open sets}
First we provide general estimates, valid for any open subset $\Omega \subset \R^d$, $d \geq 2$. Let
$$
\mathbb{S}^{d-1} \, = \, \left\{ x \in \R^d \, : \, |x| = 1 \right\}
$$
denote the unit-sphere in $\R^d$. For an arbitrary direction $u \in \mathbb{S}^{d-1}$ and $x \in \Omega$ set
\begin{eqnarray*}
\theta(x,u) & = &  \inf \left\{ t>0 \, : \, x + tu \notin \Omega \right\} \, , \\
d(x,u) & =  & \inf \{ \theta(x,u), \theta(x,-u) \} \quad \textnormal{and} \\
l(x,u) & = & \theta(x,u) + \theta(x,-u) \, .
\end{eqnarray*}

\begin{theorem}
\label{thm:general}
Let $\Omega \subset \R^d$ be an open set and let $u \in \mathbb{S}^{d-1}$ and $\sigma \geq 3/2$. Then for all $\La > 0$ the estimate 
\begin{equation}
\label{eq:thm:general}
\tr \lk -\Delta - \La \rk^\sigma_- \, \leq \, L^{cl}_{\sigma,d} \int_\Omega \lk \La - \frac{1}{4 \,  d(x,u)^2} \rk_+^{\sigma+d/2} \, dx
\end{equation}
holds true.
\end{theorem}

\begin{remark}
Let us define
\begin{equation}
\label{eq:lineseg}
l_0 \, = \, \inf_{u \in \mathbb{S}^{d-1}} \sup_{x \in \Omega} \, l(x,u) \, .
\end{equation}
Then Theorem \ref{thm:general} implies the following improvement of Melas-type: For $\sigma \geq 3/2$ and all $\La > 0$ the estimate
\begin{equation}
\label{in:simple}
\tr \lk -\Delta - \La \rk^\sigma_- \, \leq \, L^{cl}_{\sigma,d} \, |\Omega|   \lk  \La - \frac{1}{l_0^2} \rk_+^{\sigma+d/2}
\end{equation}
holds.
In convex domains $l_0$ is the minimal width of the domain, see \cite{BonFen48}. In this case $1/l_0^2$ is bounded from below by a multiple of $|\Omega|^{-2/d}$, \cite{YagBol61}, while no such bound holds for the improving term $|\Omega|/I(\Omega)$ in Melas' inequality (\ref{eq:melas}).
\end{remark}

The proof of Theorem \ref{thm:general} relies on a lifting technique, which was introduced in \cite{Laptev97}, see also \cite{LapWei00, ExLiWe04,Weidl08,FraLap08} for further developments and applications.

\begin{proof}[Proof of Theorem \ref{thm:general}]
We apply the argument used in \cite{Weidl08} to reduce the problem to  one-dimensional estimates. Fix a Cartesian coordinate system in $\R^d$, such that the given direction $u$ corresponds to the vector $(0,\dots,0,1)$.

For $x \in \R^d$ write $x = (x',t) \in \mathbb{R}^{d-1} \times \mathbb{R}$ and let $\nabla'$ and $-\Delta'$ denote the gradient and the Laplace operator in the first $d-1$ dimensions. Each section $\Omega(x') = \{t \in \mathbb{R} : (x',t) \in \Omega \}$ consists of at most countably many open intervals $J_k(x') \subset \R$, $k =1,\dots,N(x') \leq \infty$. 

We consider the quadratic form of $-\Delta - \La$ on functions $\varphi$ from the form core $C_0^\infty(\Omega)$ and write 
\[
\left\| \nabla \varphi \right\|^2_{L^2(\Omega)} - \La \left\|\varphi \right\|^2_{L^2(\Omega)} \,= \,\left\| \nabla' \varphi \right\|^2_{L^2(\Omega)} + \int_{\R^{d-1}} dx' \int_{\Omega(x')} \lk \left| \partial_t \varphi \right|^2 - \La |\varphi|^2 \rk dt \, .
\]
The functions $\varphi(x',\cdot)$ satisfy Dirichlet boundary conditions at the endpoints of each interval $J_k(x')$ forming $\Omega(x')$, hence 
\begin{eqnarray*}
\int_{\Omega(x')} \lk \left| \partial_t \varphi \right|^2 - \La |\varphi|^2 \rk dt & = & \sum_{k=1}^{N(x')} \int_{J_k(x')} \lk \left| \partial_t \varphi \right|^2 - \La |\varphi|^2 \rk dt\\
& \geq & - \sum_{k=1}^{N(x')} \langle V_k(x') \varphi(x', \cdot), \varphi(x',\cdot) \rangle_{L^2 \lk J_k(x') \rk} \, ,
\end{eqnarray*}
where the bounded, non-negative operators $V_k(x',\La) = (-\partial_t^2-\La)_-$ are the negative parts of the Sturm-Liouville operators $-\partial_t^2 - \La$ with 
Dirichlet boundary conditions on $J_k(x')$. Let 
$$
V(x',\La)= \bigoplus_{k=1}^{N(x')} V_k(x',\La)
$$
be the negative part of $-\partial_t^2 - \La$ on $\Omega(x')$ subject to Dirichlet boundary conditions on the endpoints of each interval $J_k(x')$,
$k=1,\dots,N(x')$, that is on $\partial\Omega(x')$.
Then 
\[
\int_{\Omega(x')} \lk \left| \partial_t \varphi \right|^2 - \La |\varphi|^2 \rk dt \, \geq \, - \langle V(x',\Lambda) \varphi(x',\cdot ), \varphi(x', \cdot ) \rangle_{L^2(\Omega(x'))}
\]
and consequently
\begin{equation*}
\left\| \nabla \varphi \right\|^2_{L^2(\Omega)} - \La \left\|\varphi \right\|^2_{L^2(\Omega)} \, \geq \,  \left\| \nabla' \varphi \right\|^2_{L^2(\Omega)} - \int_{\R^{d-1}} dx' \, \langle V \varphi(x',\cdot ), \varphi(x', \cdot ) \rangle_{L^2(\Omega(x'))}  .
\end{equation*}
Now we can extend this quadratic form by zero to $C_0^\infty \lk \R^d \setminus \partial \Omega \rk$, which is a form core for $( - \Delta_{\R^d \setminus \Omega} ) \oplus \lk -\Delta_\Omega - \La \rk$. This operator corresponds to the left hand side of the equality above, while the semi-bounded form on the right hand side is closed on the larger domain $H^1 \lk \R^{d-1},L^2(\R) \rk$, where it corresponds to the operator 
\begin{equation}
\label{auxop'}
-\Delta' \otimes \mathbb{I} - V(x',\La)\quad\mbox{on}\quad L^2 \lk \R^{d-1} , L^2(\R) \rk\,. 
\end{equation}
Due to the positivity of $-\Delta_{\R^d \setminus \Omega}$ we can use the variational principle to deduce that for any $\sigma\geq 0$
\begin{eqnarray*}
\textnormal{Tr} \lk -\Delta_\Omega - \La \rk^{\sigma}_{-} & = & \textnormal{Tr} \lk \lk -\Delta_{\R^d \setminus \Omega} \rk  \oplus \lk - \Delta_\Omega - \La \rk \rk^{\sigma}_{-}\\
& \leq & \textnormal{Tr} \lk - \Delta' \otimes \mathbb{I} - V(x',\La) \rk_-^{\sigma}.
\end{eqnarray*}
Now we apply sharp Lieb-Thirring inequalities \cite{LapWei00} to the Schr\"odinger operator (\ref{auxop'})
with the operator-valued potential $-V(x',\La)$ and obtain that for $\sigma \geq 3/2$
\begin{equation}
\label{in:basicproof}
\textnormal{Tr} \lk - \Delta_\Omega - \Lambda \rk_-^{\sigma} \,\leq \,L^{cl}_{\sigma,d-1} \int_{\R^{d-1}} 
\textnormal{Tr}\, V(x',\La)^{\sigma+{(d-1)/2}} \,dx' \, .
\end{equation}

To estimate the trace of the one-dimensional differential operator $V(x',\La)$ we apply Lemma \ref{lem:1d}. Our choice of coordinate system implies that for $x = (x',t) \in J_k(x')$ the distance of $t$ to the boundary of the interval $J_k(x')$ is given by $d(x,u)$. Hence, Lemma \ref{lem:1d} implies
\begin{eqnarray*}
\tr V(x',\La)^{\sigma+(d-1)/2} & = & \sum_{k=1}^{N(x')} \tr \lk \left. - \frac{d^2}{dt^2} \right|_{J_k(x')} - \La \rk_-^{\sigma+(d-1)/2} \\
& \leq & L^{cl}_{\sigma+(d-1)/2,1} \int_{\Omega(x')} \lk \La - \frac{1}{4 \, d((x',t),u)} \rk_+^{\sigma+d/2} \, dt 
\end{eqnarray*}
and the result follows from (\ref{in:basicproof}) and the identity $L^{cl}_{\sigma,d-1} \, L^{cl}_{\sigma+(d-1)/2,1} = L^{cl}_{\sigma,d}$.
\end{proof}

We proceed to analysing the geometric properties of (\ref{eq:thm:general}). Note that the left hand side of (\ref{eq:thm:general}) is independent of the choice of direction $u \in \mathbb{S}^{d-1}$, while the right hand side depends on $u$ and therefore on the geometry of $\Omega$. For a given set $\Omega$ one can minimise the right hand side in $u \in \mathbb{S}^{d-1}$. The result, however, depends on the geometry of $\Omega$ in a rather tricky way. In order to make this geometric dependence more transparent, we average the right hand side of (\ref{eq:thm:general}) over $u \in \mathbb{S}^{d-1}$. Even though the resulting bound is - in general - not as precise as (\ref{eq:thm:general}), it allows a more appropriate geometric interpretation.

To analyse the effect of the boundary, one would like to estimate $d(x,u)$ in terms of the distance to the boundary, see \cite{Davies95,Davies99,HoHoLa02}, where this approach is used to derive geometrical versions of Hardy's inequality. To avoid complications that arise, for example, if the complement of $\Omega$ contains isolated points, we use slightly different notions: For $x \in \Omega$ let
$$
\Omega(x) \, = \, \left\{ y \in \Omega \, : \, x + t(y-x) \in \Omega\, , \ \forall \, t \in [0,1] \right\}
$$
be the part of $\Omega$ that "can be seen" from $x$ and let
$$
\delta(x) \, = \, \inf \left\{ |y-x| \, : \, y \notin \overline{\Omega(x)} \right\}
$$
denote the distance to the exterior of $\Omega(x)$.

For fixed $\varepsilon > 0$ put 
$$
A_\varepsilon(x) \, = \, \left\{ a \in \R^d \setminus \overline{\Omega(x)} \, : \, |x-a| < \delta(x) + \varepsilon \right\}
$$
and for $a \in A_\varepsilon(x)$ set $B_x(a) = \{y \in \R^d \, : \, |y-a| < |x-a| \}$ and 
$$
\rho_a(x) \, = \, \frac{ | B_x(a) \setminus \overline{\Omega(x)} |}{\omega_d |x-a|^d} \, ,
$$
where $\omega_d$ denotes the volume of the unit ball in $\R^d$.
To get a result, independent of $a$ and $\varepsilon$, set
$$
\rho(x) \, = \, \inf_{\varepsilon > 0}  \, \sup_{a \in A_\varepsilon(x)} \rho_a(x) \, .
$$
Note that $\R^d \setminus \overline{\Omega(x)}$ is open, hence $\rho_a(x)>0$ and $\rho(x) >0$ hold for $x \in \Omega$. Finally, we define
$$
M_\Omega(\La) \, = \, \int_{R_\Omega(\La)} \rho(x) \, dx \, ,
$$
where $R_\Omega(\La) \subset \Omega$ denotes the set $\{ x \in \Omega \, : \, \delta(x) < 1/(4\sqrt\La)\}$. The main result of this section allows a geometric interpretation of the remainder term:

\begin{theorem}
\label{thm:gen}
Let $\Omega \subset \R^d$ be an open set with finite volume and $\sigma \geq 3/2$. Then for all $\La > 0$ we have
\begin{equation}
\label{eq:thm:gen}
\tr \lk -\Delta - \La \rk^\sigma_- \, \leq \, L^{cl}_{\sigma,d} \, |\Omega| \La^{\sigma+d/2} - L^{cl}_{\sigma,d} \, 2^{-d+1} \,  \La^{\sigma+d/2} \, M_\Omega(\La) \, .
\end{equation}
\end{theorem}

The function $\rho(x)$ depends on the behaviour of the boundary close to $x \in \Omega$. For example, $\rho(x)$ is small close to a cusp. On the other hand $\rho(x)$ is larger than $1/2$ in a convex domain. By definition, the function $M_\Omega(\Lambda)$ gives an average of this behaviour over $R_\Omega(\Lambda)$, which is like a tube of width $1/(4\sqrt \La)$ around the boundary. 

Note that $M_\Omega(\La)$ tends to zero as $\La$ tends to infinity. This decay in $\La$ is of the order $(\delta_M - d)/2$,  where $\delta_M$ denotes  the interior Minkowski dimension of the boundary, see e.g. \cite{Lapidu91,FleVas93,FlLeVa95} for definition and examples. If $d-1 \leq \delta_M < d$ and if the upper Minkowski content of the boundary is finite, then the second term of the asymptotic limit of the Riesz means equals $O\lk\La^{\sigma+\delta_M/2}\rk$ as $\La \to \infty$, see \cite{Lapidu91}. Therefore the remainder term in (\ref{eq:thm:gen}) reflects the correct order of growth in the asymptotic limit.

In particular, if the dimension of the boundary equals $d-1$, we find
$$
M_\Omega(\La) = |\partial \Omega| \, \La^{-1/2} + o(\La^{-1/2})
$$
as $\La \to \infty$ and the second term in (\ref{eq:thm:gen}) is in close  correspondence with the asymptotic formula (\ref{eq:asymptotics}).

\begin{proof}[Proof of Theorem \ref{thm:gen}]
We start from the result of Theorem \ref{thm:general} and average over all directions to get
\begin{equation}
\label{eq:hd_basic}
\tr \lk -\Delta - \La \rk^\sigma_- \, \leq \, L^{cl}_{\sigma,d} \, \La^{\sigma + d/2}  \int_\Omega \int_{\mathbb{S}^{d-1}} \lk 1 - \frac{1}{4 \, \La \, d(x,u)^2} \rk_+^{\sigma+d/2}  d\nu(u) \, dx \, ,
\end{equation}
where $d\nu(u)$ denotes the normed measure on $\mathbb{S}^{d-1}$.

For $x \in \Omega$ and $a \notin \overline{\Omega(x)}$ let $\Theta(x,a) \subset \mathbb{S}^{d-1}$ be the subset of all directions $u \in \mathbb{S}^{d-1}$, satisfying $x + su \in B_x(a) \setminus \overline{\Omega(x)}$ for some $s > 0$. For such $s$ we have
\begin{equation}
\label{eq:hd_ball1}
s \, \leq \, 2 \, |x-a| \, .
\end{equation}
By definition of $\rho_a(x)$ and $\Theta(x,a)$ we find
$$
\rho_a(x) \, \omega_d |x-a|^d \, = \, |B_x(a) \setminus \overline{\Omega(x)} | \, \leq \, \int_{\Theta(x,a)} d\nu(u) \, \omega_d (2|x-a|)^d \, ,
$$
hence
\begin{equation}
\label{eq:hd_ball2}
\int_{\Theta(x,a)} d\nu(u) \, \geq \, 2^{-d} \, \rho_a(x)  \, .
\end{equation}
Using (\ref{eq:hd_ball1}) we also see that for $u \in \Theta(x,a)$ the estimate $d(x,u) \leq s \leq  2 \, |x-a|$ holds.

Now fix $\La > 0$ and choose $0 < \varepsilon < 1/(4\sqrt \La)$ and $a \in A_\varepsilon(x)$. By definition of $A_\varepsilon(x)$ it follows that for all $u \in \Theta(x,a)$
\begin{equation}
\label{eq:hd_dist}
d(x,u) \, \leq \, 2|x-a| \, < \, 2 ( \delta(x) + \varepsilon ) \, .
\end{equation}
The set $\Theta(x,a)$ must be contained in one hemisphere of $\mathbb{S}^{d-1}$ which we denote by $\mathbb{S}^{d-1}_+$. Using that $d(x,u) = d(x,-u)$ we estimate
\begin{align*}
\int_{\mathbb{S}^{d-1}} \lk 1 - \frac{1}{4 \La d(x,u)^2} \rk_+^{\sigma+d/2} d\nu(u) \, &= \, 2 \int_{\mathbb{S}^{d-1}_+} \lk 1 - \frac{1}{4 \La d(x,u)^2} \rk_+^{\sigma+d/2} d\nu(u) \\
&\leq \, 1 - 2 \int_{ \left\{ u \in \mathbb{S}^{d-1}_+ \, : \, d(x,u) \leq 1/(2\sqrt \La) \right\} } d\nu(u) \, .
\end{align*}
Assume that $\delta(x) \leq 1/(4 \sqrt\La) - \varepsilon$. From (\ref{eq:hd_dist}) it follows that
$$
\Theta(x,a) \subset \left\{ u \in \mathbb{S}^{d-1}_+ \, : \, d(x,u) \leq 1/(2\sqrt\La) \right\} \, ,
$$
hence, using (\ref{eq:hd_ball2}), we conclude
$$
\int_{\mathbb{S}^{d-1}} \lk 1 - \frac{1}{4 \La d(x,u)^2} \rk_+^{\sigma+d/2} d\nu(u) \, \leq \, 1 - 2 \int_{\Theta(x,a)} d\nu(u) \, \leq \, 1 - 2^{1-d} \rho_a(x) \, .
$$
Since $a \in A_\varepsilon(x)$ was arbitrary we arrive at
$$
\int_{\mathbb{S}^{d-1}} \lk 1 - \frac{1}{4 \La d(x,u)^2} \rk_+^{\sigma+d/2} d\nu(u) \, \leq \, 1 - 2^{1-d} \rho(x) \, ,
$$
for all $x \in \Omega$ with $\delta(x) \leq 1/(4 \sqrt\La) - \varepsilon$ and we can take the limit $\varepsilon \to 0$. 

It follows that 
$$
\int_\Omega \int_{\mathbb{S}^{d-1}} \lk 1 - \frac{1}{4 \La d(x,u)^2} \rk_+^{\sigma+d/2} d\nu(u) \, dx \, \leq \,  |\Omega| - 2^{1-d} \int_{ \left\{ x \in \Omega \, : \, \delta(x) < 1/(4\sqrt\La)  \right\} } \rho(x) \, dx
$$
and inserting this into (\ref{eq:hd_basic}) yields the claimed result.
\end{proof}

\subsection{Convex domains}
If $\Omega \subset \R^d$ is convex, we have $\Omega(x) = \Omega$ and 
\begin{equation}
\label{eq:hd_rho}
\rho(x) \, > \, 1/2
\end{equation}
for all $x \in \Omega$. Thus we can simplify the remainder term, by estimating $M_\Omega(\La)$.

\begin{corollary}
\label{cor:convex}
Let $\Omega \subset \R^d$ be a bounded, convex domain with smooth boundary and assume that the curvature of $\partial \Omega$ is bounded from above by $1/R$. Then for $\sigma \geq 3/2$ and all $\La > 0$ we have 
$$
\tr \lk -\Delta - \La \rk^\sigma_- \, \leq \, L^{cl}_{\sigma,d} \, |\Omega| \La^{\sigma+d/2} - L^{cl}_{\sigma,d} \, 2^{-d-2} \, |\partial \Omega| \, \La^{\sigma+(d-1)/2} \int_0^1 \lk 1 - \frac{d-1}{4R\sqrt\La} \, s \rk_+ \! ds \, .
$$
\end{corollary}

\begin{proof}
Inserting (\ref{eq:hd_rho}) into the definition of $M_\Omega(\La)$ yields
$$
M_\Omega(\La) \, > \, \frac 12 \int_{R_\Omega(\La)} dx \, .
$$
Let $\Omega_t = \{ x \in \Omega \, : \, \delta(x) > t \}$ be the inner parallel set of $\Omega$ and write
$$
\int_{R_\Omega(\La)} dx \, = \, \int_{ \left\{ x \in \Omega \, : \, \delta(x) < 1/(4\sqrt\La)  \right\} } dx \, =  \, \int_0^{1/(4\sqrt \La)} \left| \partial \Omega_t \right|  dt \, .
$$
Now we can use Steiner's Theorem, see \cite{Guggen77,Berg84a}, namely 
\begin{equation}
\label{in:steiner}
\left| \partial \Omega_t \right| \, \geq \, \lk 1 - \frac{d-1}{R} \, t \rk_+ |\partial \Omega|  \, .
\end{equation}
It follows that
$$
M_\Omega(\La) \, > \, \frac 12 |\partial \Omega| \int_0^{1/(4\sqrt \Lambda)} \lk 1 - \frac{d-1}{R} \, t \rk_+ \, dt \, = \,  \frac{|\partial \Omega|}{8 \sqrt \La} \int_0^1 \lk 1 - \frac{d-1}{4 R \sqrt \La} \, s \rk_+ \! ds \, .
$$
Inserting this into (\ref{eq:thm:gen}) completes the proof.
\end{proof}

Let us single out the case where $\Omega = B_r$ is a ball in $\R^d$ with radius $r > 0$. Note that the first eigenvalue of the Dirichlet Laplace operator on $B_r$ is given by
$$
\lambda_1(B_r) \, = \, \frac{\pi \, j^2_{d/2-1,1}}{\Gamma \lk d/2 +1 \rk^{2/d} \, \left| B_r \right|^{2/d}} \, ,
$$
where $j_{d/2-1,1}$ denotes the first zero of the Bessel function $J_{d/2-1}$. Again, we have to estimate
$$
M_\Omega(\La) \, > \, \frac 12  \int_0^{1/(4\sqrt\La)} \left| \partial(B_r)_t \right| \, dt \, .
$$
Instead of (\ref{in:steiner}) we can now use that 
$$
\left|  \partial \lk B_r \rk_t \right| \, = \, \left| \partial B_r \right| \lk 1 - \frac tr \rk^{d-1}_+
$$
and conclude that for $\sigma \geq 3/2$ and $\La > 0$
$$
\tr \lk -\Delta - \La \rk_-^\sigma \, \leq \, L^{cl}_{\sigma,d} \, \left| B_r \right| \, \La^{\sigma + d/2} - L^{cl}_{\sigma,d} \, 2^{-d-2} \left| \partial B_r \right| \int_0^1 \lk 1- \frac{s}{4r \sqrt\La} \rk_+^{d-1} \! ds  \, \La^{\sigma+(d-1)/2} \, .
$$
We note that $\tr \lk -\Delta -\La \rk_-^\sigma = 0$ for $\La \leq \lambda_1(B_r)$. For $\La > \lambda_1(B_r)$ we apply the foregoing inequality with $\sigma = 3/2$. Then the method from \cite{AizLie78} yields that for $\sigma \geq 3/2$ the inequality
\begin{equation}
\label{eq:hd_ball}
\tr \lk -\Delta - \La \rk_-^\sigma \, \leq \, L^{cl}_{\sigma,d} \, \left| B_r \right| \, \La^{\sigma + d/2} - C_{ba} \, L^{cl}_{\sigma,d-1} \, \left| \partial B_r \right| \, \La^{\sigma+(d-1)/2} 
\end{equation}
holds with a constant
$$
C_{ba} \, = \, \frac{j_{d/2-1,1}}{2^{d+1} \, d \,  \pi^{1/2}} \, \frac{\Gamma \lk (d+4)/2 \rk} {\Gamma \lk (d+5)/2 \rk } \lk 1 - \lk 1 - \frac{1}{4 j_{d/2-1,1}} \rk^d \rk \, > \, 0 \, .
$$

%%%%%%%%%%%%%%%%%%%%%%%%%%%%%%%%%%%%%%%%%%%%%%%%%%%%%%%%%%%%%%%%%%%%%%%%%%%%%%%%%%%%%%%%%%%%%

\section{Lower bounds on individual eigenvalues}
\label{sec:lower}

In order to further estimate the remainder terms, in particular to show that the remainder is negative for all $\La \geq \lambda_1(\Omega)$ - as in (\ref{eq:hd_ball}) for the ball - one needs suitable bounds on the ground state $\lambda_1(\Omega)$.  We point out the following consequence of the proof of Theorem \ref{thm:general} which might be of independent interest.

\begin{corollary}
\label{cor:gdstate}
For any open set $\Omega \subset \R^d$ the estimate
$$
\lambda_1(\Omega) \, \geq \, \frac{\pi^2}{l_0^2}
$$
holds, where $l_0$ is given in (\ref{eq:lineseg}).
\end{corollary}

\begin{proof}
Fix $\varepsilon > 0$ and choose a direction $u_0 \in \mathbb{S}^{d-1}$, such that $\sup_{x \in \Omega} l(x,u_0) < l_0 + \varepsilon$. We write $x = (x',t) \in \R^{d-1} \times \R$, where the $t$-axes is chosen in the direction of $u_0$. Let us recall inequality (\ref{in:basicproof}) from the proof of Theorem \ref{thm:general}: For any $\sigma \geq 3/2$
$$
\textnormal{Tr} \lk - \Delta - \Lambda \rk_-^{\sigma} \,\leq \,L^{cl}_{\sigma,d-1} \int_{\R^{d-1}} 
\textnormal{Tr}\, V(x',\La)^{\sigma+{(d-1)/2}} \,dx' \, ,
$$
where $V(x',\La)$ denotes the negative part of the operator $-\partial_t^2 - \La$ on $\Omega(x') = \bigcup_{k=1}^{N(x')} J_k(x')$ with Dirichlet boundary conditions at the endpoints of each interval $J_k(x')$. This inequality can be rewritten as
$$
\textnormal{Tr} \lk - \Delta - \Lambda \rk_-^{\sigma} \,\leq \,L^{cl}_{\sigma,d-1} \int_{\R^{d-1}} \sum_{k=1}^{N(x')} \sum_{j \in \mathbb{N}} \lk \La - \frac{\pi^2 \, j^2}{|J_k(x')|^2} \rk_+^{\sigma+(d-1)/2} \, dx' \, .
$$
Our choice of coordinate system implies $|J_k (x')| \leq \sup_{x \in \Omega} l(x,u_0) < l_0 + \varepsilon$ for all $k = 1, \dots, N(x')$ and all $x' \in \R^{d-1}$. It follows that the right hand side of the inequality above is zero for all $\La \leq \pi^2/(l_0+\varepsilon)^2$. Thus by taking the limit $\varepsilon \to 0$ we find
$$
\sum_{n \in \mathbb{N}} \lk \La - \lambda_n \rk_+^\sigma \, = \, \textnormal{Tr} \lk - \Delta - \Lambda \rk_-^{\sigma} \, = \,   0 \, ,
$$
for all $\La \leq \pi^2/l_0^2$ and $\lambda_1 \geq \pi^2 / l_0^2$ follows.
\end{proof}

From (\ref{in:simple}) we obtain similar bounds on higher eigenvalues using a method introduced in \cite{Laptev97}.

\begin{corollary}
\label{cor:Nsimple}
For any open set $\Omega \subset \R^d$ with finite volume and any $k \in \mathbb{N}$ the estimate
$$
\lambda_k(\Omega) \, \geq \, C_d \, \lk \frac{12}{\pi} \rk^{1/d} \frac{d}{(d+3)^{1+1/d}} \lk \frac{\Gamma \lk (d+3)/2 \rk}{\Gamma(d/2+1)} \rk^{2/d} \, \frac{k^{2/d}}{|\Omega|^{2/d}}  + \frac{1}{l_0^2}
$$
holds, with
$$
C_d \, = \, 4  \pi \, \Gamma \lk d/2+1 \rk^{2/d} \, .
$$
\end{corollary}

\begin{proof}
Let $N(\La) = \tr \lk -\Delta - \La \rk_-^0$ denote the counting function of the eigenvalues below $\La > 0$. In \cite{Laptev97} it is shown that for $\sigma > 0$, and all $\La >0$, $\tau >0$
\begin{equation}
\label{in:numberest}
N(\La) \, \leq \, \lk \tau  \La \rk^{-\sigma} \, \tr \lk -\Delta - (1 + \tau) \La \rk_-^\sigma  \, . 
\end{equation}
If we apply this inequality with $\sigma = 3/2$, we can use (\ref{in:simple}) to estimate
$$
N(\La) \, \leq \, L^{cl}_{3/2,d} \, |\Omega| \, \La^{d/2} \, \frac{(1+\tau)^{(d+3)/2}}{\tau^{3/2}} \lk 1 - \frac{1}{\La (1+\tau) l_0^2} \rk_+^{(d+3)/2} \, .
$$
Minimising the right hand side in $\tau > 0$ yields $\tau_{\min} = 3(\La l^2_0-1)/(d \La l_0^2 ) $ and inserting this we find
$$
N(\La) \, \leq \, L^{cl}_{3/2,d} \, |\Omega| \, \frac{(d+3)^{(d+3)/2}}{3^{3/2} \,  d^{d/2}} \lk \La - \frac{1}{l_0^2} \rk^{d/2}_+ \, .
$$
This is equivalent to the claimed result.
\end{proof}

\begin{remark}
Applying the same method to (\ref{in:beliyau}) with $\sigma = 1$ we recover the Li-Yau inequality (\ref{in:liyau}).
In the proof of Corollary \ref{cor:Nsimple} we have to start from $\sigma = 3/2$, therefore the result is not strong enough to improve (\ref{in:liyau}) in general, but one gets improvements for low eigenvalues whenever $l_0$ is small.
\end{remark}

%%%%%%%%%%%%%%%%%%%%%%%%%%%%%%%%%%%%%%%%%%%%%%%%%%%%%%%%%%%%%%%%%%%%%%%%%%%%%%%%

\section{Further improvements in dimension 2}
\label{sec:2d}

In this section we further improve Corollary \ref{cor:convex} and generalise it to a large class of bounded convex domains $\Omega \subset \R^2$. Here we do not require smoothness, therefore we cannot use (\ref{in:steiner}) to estimate inner parallels of the boundary. To find a suitable substitute let $w$ denote the minimal width of $\Omega$ and note that for $l_0$ given in (\ref{eq:lineseg}) the identity
$$
w \, = \, l_0
$$
holds true, see e.g. \cite{BonFen48}. In the remainder of this section we assume that for all $t > 0$
\begin{equation}
\label{eq:innerwidth}
|\partial \Omega_t| \, \geq \, \lk 1 - \frac{3t}{w} \rk_+ |\partial \Omega| \, .
\end{equation}
This is true for a large class of convex domains, including the circle, regular polygons and arbitrary triangles. Actually we conjecture that (\ref{eq:innerwidth}) holds true for all bounded convex domains in $\R^2$.

Furthermore we need a lower bound on the ground state. From Corollary \ref{cor:gdstate} we obtain that for all convex domains $\Omega \subset \R^2$
\begin{equation}
\label{eq:gdstate}
\lambda_1 (\Omega) \, \geq \, \frac{\pi^2}{w^2} 
\end{equation}
holds. One should mention, that the same estimate can be obtained from the inequality $\lambda_1 (\Omega) \geq \pi^2 / (4 r_{in}^2)$, see \cite{osserm77}, where $r_{in}$ is the inradius of $\Omega$.

Using similar but more precise methods as in the proof of Theorem \ref{thm:gen} we get

\begin{theorem}
\label{thm:convex}
Let $\Omega \subset \R^2$ be a bounded, convex domain, satisfying (\ref{eq:innerwidth}).
Then for $\sigma \geq 3/2$ we have 
\begin{align*}
\tr \lk -\Delta - \La \rk_-^\sigma \, &= \, 0 &\textnormal{if} \ & \La \leq \pi^2/w^2 \ \textnormal{and} \\
\tr \lk -\Delta - \La \rk_-^\sigma \, &\leq \, L^{cl}_{\sigma,2} \, |\Omega| \, \La^{\sigma+1} - C_{co} \, L^{cl}_{\sigma,1} \, |\partial \Omega| \, \La^{\sigma+1/2} & \textnormal{if} \ & \La > \pi^2/w^2 \ ,
\end{align*}
with a constant
$$
C_{co} \, \geq \, \frac{11}{9 \pi^2} - \frac{3}{20 \pi^4} - \frac{2}{5\pi^2} \ln \lk \frac{4 \pi}{3} \rk \, > \, 0.0642 \, .
$$
\end{theorem}

\begin{proof}
The first claim follows directly from (\ref{eq:gdstate}), thus we can assume $\La > \pi^2/w^2$.
First we prove the result for $\sigma = 3/2$.
Again we can start from (\ref{eq:hd_basic}) and we need to estimate $d(x,u)$ in terms of $\delta(x)$, which is just the distance to the boundary, since $\Omega$ is convex.

Fix $x \in \Omega$. Since $\Omega$ is convex and smooth we can choose $u_0 \in \mathbb{S}^{d-1}$, such that $d(x,u_0) = \delta(x)$. We can assume $u_0 = (1,0,\dots,0)$ and put $\mathbb{S}^{d-1}_+ = \left\{u \in \R^d \, : \, u_1 > 0 \right\}$.

Let $a$ be the intersection point of the semi-axes $\left\{ x+tu_0 , t > 0 \right\}$ with $\partial \Omega$ and for arbitrary $u \in \mathbb{S}^{d-1}_+$ let $b_u$ be the intersection point of $\left\{x+tu,t>0\right\}$ with the plane through $a$, orthogonal to $u_0$. We find $d(x,u) \leq |x-b_u|$ and if $\theta_u$ denotes the angle between $u_0$ and $u$, we find
$$
d(x,u) \, \leq \, |x-b_u| \, = \, \frac{|x-a|}{\cos \theta_u} \, = \, \frac{\delta(x)}{\cos \theta_u} \, .
$$
Using (\ref{eq:hd_basic}) and taking into account that $d(x,u) = d(x,-u)$ we can estimate
\begin{eqnarray}
\nonumber
\tr \lk -\Delta - \La \rk_-^{3/2} & \leq &  L^{cl}_{3/2,2} \, \La^{5/2} \int_\Omega \frac{2}{\pi} \int_{\theta_0}^{\pi/2} \lk 1 - \frac{\cos^2 (\theta)}{4 \La \delta(x)^2} \rk^{5/2} d\theta \, dx\\
\label{in:2dint}
& \leq & \frac{\La^{5/2}}{5 \pi^2} \int_\Omega \int_{\theta_0}^{\pi/2} \lk 1 - \frac{\cos^2 (\theta)}{2 \La \delta(x)^2} + \frac{\cos^4 (\theta)}{16 \La^2 \delta(x)^4} \rk d\theta   \, dx  ,
\end{eqnarray}
where $\theta_0 = 0$ if $\delta(x) \geq 1/(2 \sqrt \La)$ and $\theta_0 = \arccos (2\delta(x)\sqrt \La)$ if $\delta(x) < 1/(2 \sqrt \La)$. We set $u_0 = \min(1,2 \delta(x) \sqrt \La)$ and calculate 
\begin{eqnarray*}
\lefteqn{ \int_{\theta_0}^{\pi/2} \lk 1 - \frac{\cos^2 (\theta)}{2 \La \delta(x)^2} + \frac{\cos^4 (\theta)}{16 \La^2 \delta(x)^4} \rk d\theta } \\
& = & \frac \pi2 - \theta_0 - \frac{\arcsin(u_0) - u_0 \sqrt{1-u_0^2} }{4 \La \delta(x)^2} + \frac{3 \arcsin(u_0) - u_0 \lk 2 u_0^2+3 \rk \sqrt{1-u_0^2}}{128 \La^2 \delta(x)^4}  \, .
\end{eqnarray*}
Inserting this back into (\ref{in:2dint}) yields
\begin{equation}
\label{in:2dbasic}
\tr \lk -\Delta - \La \rk_-^{3/2} \, \leq \, \frac{1}{10 \pi} \, |\Omega| \, \La^{5/2} - \La^{5/2} \lk I_1 + I_2 \rk \, ,
\end{equation}
where
$$
I_1 \, = \, \int_{ \{ \delta(x) \geq 1/(2\sqrt\La) \} } \frac{1}{10\pi} \lk \frac{1}{4 \La \delta(x)^2} - \frac{3}{128\La^2\delta(x)^4} \rk dx
$$
and
\begin{eqnarray*}
I_2 & = & \int_{ \{ \delta(x) < 1/(2\sqrt\La) \} } \frac{1}{5\pi^2} \lk \arccos(2 \sqrt \La \delta(x)) + \frac{\arcsin(2\sqrt \La \delta(x))}{4\La\delta(x)^2} - \frac{ \sqrt{1-4 \La \delta(x)^2} }{2\sqrt\La \delta(x)}   \right. \\
&& \left. -\frac{3 \arcsin(2\sqrt\La \delta(x))}{128 \La^2 \delta(x)^4} + \frac{(8 \La \delta(x)^2+3) \sqrt{1-4 \La \delta(x)^2}}{64 \La^{3/2} \delta(x)^3} \rk dx \, .
\end{eqnarray*}

First we turn to 
$$
I_1 \, = \, \frac{1}{10 \pi} \int_{1/(2\sqrt \La)}^\infty \left| \partial \Omega_t \right| \lk \frac{1}{4 \La t^2} - \frac{3}{128\La^2t^4} \rk dt \, .
$$
Note that the term in brackets is positive, thus after substituting $s= 2 \sqrt \La t$ we can use (\ref{eq:innerwidth}) and $\La > \pi^2/w^2$ to obtain
\begin{eqnarray*}
I_1 & \geq & \frac{1}{20 \pi } \frac{|\partial \Omega|}{\sqrt \La} \int_1^\infty \lk 1 - \frac{3s}{2 \pi} \rk_+ \lk \frac{1}{s^2} - \frac{3}{8s^4} \rk ds\\
& = &  \frac{1}{20 \pi } \frac{|\partial \Omega|}{\sqrt \La} \lk \frac 78 - \frac{39}{32 \pi} - \frac{27}{128 \pi^3} - \frac{3}{2\pi} \ln \lk \frac{2\pi}{3} \rk \rk \, .
\end{eqnarray*}

Similarly we can treat $I_2$ and get
$$
I_2 \, \geq \, \frac{1}{10 \pi^2} \frac{|\partial \Omega|}{\sqrt \La} \lk \frac{557}{192} - \frac{7 \pi}{16} - \frac{3}{2 \pi} \int_0^1 \frac{\arcsin(s) }{s} ds \rk \, .
$$
In view of $L^{cl}_{3/2,1} = 3/16$ we can write
$$
I_1 + I_2 \, \geq \, L^{cl}_{3/2,1} \frac{|\partial \Omega|}{\sqrt \La} \lk \frac{11}{9 \pi^2} - \frac{3}{20 \pi^4} - \frac{2}{5\pi^2} \ln \lk \frac{2 \pi}{3} \rk - \frac{2}{5 \pi^2} \ln(2) \rk
$$
and inserting this into (\ref{in:2dbasic}) yields the claim in the case $\sigma = 3/2$.

To prove the estimate for $\sigma > 3/2$ we again refer to \cite{AizLie78} and use the identity
$$
\tr \lk -\Delta - \La \rk_-^\sigma \, = \, \frac{1}{B(\sigma-3/2,5/2)} \int_0^\infty \tau^{\sigma-5/2} \, \tr \lk -\Delta - (\La-\tau) \rk_-^{3/2} \, d\tau \, ,
$$
from which the general result follows.
\end{proof}

Now we can apply the same arguments that lead to Corollary \ref{cor:Nsimple} to derive lower bounds on individual eigenvalues.

\begin{corollary}
\label{cor:Nconvex}
Let $\Omega \subset \R^2$ be a bounded, convex domain, satisfying (\ref{eq:innerwidth}). Then for $k \in \mathbb{N}$ and any $\alpha \in (0,1)$ the estimate
$$
\frac{\lambda_k(\Omega)}{1-\alpha} \geq 10 \pi  \alpha^{3/2} \frac{k}{|\Omega|} + \frac{15 \pi C_{co}}{8} \frac{|\partial \Omega|}{|\Omega|} \sqrt{  10 \pi  \alpha^{3/2} \frac{k}{|\Omega|} + \frac{225 \pi^2  C_{co}^2}{256} \frac{|\partial \Omega|^2}{|\Omega|^2}  } + \frac{225 \pi^2  C_{co}^2}{128} \frac{|\partial \Omega|^2}{|\Omega|^2}   
$$
holds, with the constant $C_{co}$ given in Theorem \ref{thm:convex}.
\end{corollary}

\begin{proof}
Applying (\ref{in:numberest}) and Theorem \ref{thm:convex} with $\sigma = 3/2$ yields
$$
N(\La) \, \leq \, L^{cl}_{3/2,2} \, |\Omega| \, \La \, \frac{(1+\tau)^{5/2}}{\tau^{3/2}} - C_{co} \, L^{cl}_{3/2,1} \, |\partial \Omega| \, \sqrt \La \, \frac{(1+\tau)^2}{\tau^{3/2}}
$$
for any $\tau > 0$ and $\La > \pi^2/w^2$. With $\tau = \alpha/(1-\alpha)$, $\alpha \in (0,1)$, this is equivalent to the claimed estimate.
\end{proof}

\begin{remark}
Given a fixed ratio $|\partial \Omega| / |\Omega|$ one can optimise the foregoing estimate in $\alpha \in (0,1)$, depending on $k \in \mathbb{N}$. As mentioned in the remark after Corollary \ref{cor:Nsimple}, the result cannot improve the Li-Yau inequality (\ref{in:liyau}) in general, since we have to apply (\ref{in:numberest}) with $\sigma = 3/2$ instead of $\sigma = 1$. 
However, the estimates obtained from Corollary \ref{cor:Nconvex} are stronger than (\ref{in:liyau}) for low eigenvalues and the improvements depend on the ratio $|\partial \Omega| / |\Omega|$.

In particular, one can use the isoperimetric inequality, namely that $|\partial \Omega| \geq 2 (\pi |\Omega|)^{1/2}$ for all $\Omega \subset \R^2$, to derive general improvements of the Li-Yau inequality (\ref{in:liyau}) for low eigenvalues. Indeed, from (\ref{in:liyau}) we get
$$
\lambda_k(\Omega) \, \geq \,  \frac{2\pi k}{|\Omega|} \, ,
$$
while optimising the estimate of Corollary \ref{cor:Nconvex} with $|\partial \Omega| = 2 (\pi |\Omega|)^{1/2}$, we find that for any convex domain, satisfying (\ref{eq:innerwidth}),
$$
\lambda_2 (\Omega) \, > \, \frac{15.03}{|\Omega|} \, > \, \frac{4 \pi}{|\Omega|} \ , \quad \lambda_3 (\Omega) \, > \, \frac{21.52}{|\Omega|} \, > \, \frac{6 \pi}{|\Omega|} \ , \quad \dots \ , \quad \lambda_{23}(\Omega) \, > \, \frac{144.58}{|\Omega|} \, > \, \frac{46 \pi}{|\Omega|} \, .
$$
In this way we can improve (\ref{in:liyau}) in convex domains for all eigenvalues $\lambda_k(\Omega)$ with $k \leq 23$.
\end{remark}

Finally let us make a remark about the square $Q_l = (0,l) \times (0,l) \subset \R^2$, $l > 0$. Using the methods introduced in section \ref{sec:1d} one can establish the following two-dimensional version of Lemma \ref{lem:1d}: Choose a coordinate system $(x_1,x_2) \in \R^2$ with axes parallel to the sides of the square and for $x \in Q_l$ put
$$
\delta(x_i) \, = \, \min (x_i, l - x_i) \, , \quad i = 1,2 \, .
$$
Then for $\sigma \geq 1$ and all $\Lambda > 0$ the estimate
\begin{eqnarray*} 
\tr \lk -\Delta - \La \rk_-^\sigma & = & \sum_{m,n \in \mathbb{N}} \lk \La - \frac{\pi^2}{l^2} \lk n^2 + m^2 \rk \rk_+^\sigma \\
& \leq & L^{cl}_{\sigma,2} \int_0^l \int_0^l \lk \La - C_{sq} \lk \frac{1}{\delta(x_1)} + \frac{1}{\delta(x_2)} \rk^2 \rk_+^{\sigma+1} \, dx_1 \, dx_2
\end{eqnarray*}
holds with a constant $C_{sq} > 1/10$.

%%%%%%%%%%%%%%%%%%%%%%%%%%%%%%%%%%%%%%%%%%%%%%%%%%%%%%%%%%%%%%%%%%%%%%%%%%%%%%%%%%%%%%%%%%%%%%%%%%%

\appendix

\section{Proof of Lemma \ref{lem:elementary} and Lemma \ref{lem:asympt}}

\subsection{Proof of Lemma \ref{lem:elementary}}

For $A \in \R$ let $\overline{A}$ and $\tilde{A}$ denote the integer and fractional part of $A$ respectively. Then we can calculate
\begin{eqnarray*}
\sum_k \lk 1-\frac{k^2}{A^2} \rk_+ & = & \sum_{k=1}^{\overline A} \lk 1- \frac{k^2}{A^2} \rk \, = \, \overline A - \frac{1}{A^2} \lk \frac{\overline A^3}{3} + \frac{\overline A^2}{2} + \frac{\overline A}{6} \rk \\
& = & \frac{2A}{3} - \frac 12 - \frac{1}{6 A} + \tilde A (1- \tilde A ) \frac 1A + \tilde A  \lk 1- 3 \tilde A + 2 \tilde A^2 \rk \frac{1}{6A^2} \, .
\end{eqnarray*}
From $0 \leq \tilde A < 1$ we conclude $\tilde A (1- \tilde A) \leq 1/4$ and $\tilde A \lk 1- 3 \tilde A + 2 \tilde A^2 \rk \leq \sqrt3/18$ and we get
\begin{equation}
\label{in:sum}
\sum_k \lk 1-\frac{k^2}{A^2} \rk_+ \, \leq \, \frac{2A}{3} - \frac 12 + \frac{1}{12 A} + \frac{\sqrt 3}{108A^2} \, .
\end{equation}

To estimate the right hand side of (\ref{in:elementary}) note that
$$
\int \lk 1- \frac{1}{s^2} \rk^{3/2} ds \, =  \, \lk 1 + \frac{1}{2s^2} \rk \sqrt{ s^2-1 } + \frac 32 \, \arctan \lk \frac{1}{\sqrt{s^2-1}} \rk \, , 
$$
thus
$$
\frac{2}{3\pi} \int_1^{\pi A} \lk 1-\frac{1}{s^2} \rk^{3/2} ds \, = \,  \frac{2 \pi^2 A^2 +1 }{3 \pi^2 A} \frac{\sqrt{\pi^2 A^2-1}}{\pi A}   + \frac{1}{\pi} \, \arctan \lk \frac{1}{\sqrt{\pi^2 A^2 -1}} \rk - \frac 12 \, .
$$
Now we can insert the elementary estimates
$$
\arctan \lk \frac{1}{\sqrt{\pi^2 A^2 -1}} \rk \, \geq \, \frac{1}{\pi A} \qquad \textnormal{and} \qquad 
\frac{\sqrt{\pi^2 A^2 -1 }}{\pi A} \, \geq \, 1 - \frac{1}{2 \pi^2 A^2} - \frac{a}{A^4}
$$
both valid for $A \geq 2$, where we write $a = 16 - 2/\pi^2 - 8\sqrt{4 \pi^2-1}/\pi$ for simplicity. We get
\begin{equation}
\label{in:integral}
\frac{2}{3\pi} \int_1^{\pi A} \lk 1-\frac{1}{s^2} \rk^{3/2} ds \, \geq \, \frac{2A}{3}-\frac 12 + \frac{1}{\pi^2 A} - \frac{1}{6\pi^4A^3} -\frac a3 \lk \frac{2}{A^3} + \frac{1}{\pi^2 A^5} \rk
\end{equation}
for all $A \geq 2$.
From (\ref{in:sum}) and (\ref{in:integral}) we deduce that (\ref{in:elementary}) holds true for all $A \geq 2$, since
$$
\lk \frac{1}{\pi^2} - \frac{1}{12} \rk A^4 - \frac{\sqrt 3}{108} A^3 - \lk \frac{1}{6 \pi^4} + \frac{2a}{3} \rk A^2 - \frac{a}{3 \pi^2} \, \geq \, 0
$$
for all $A \geq 2$. 

Note that (\ref{in:elementary}) is trivial for $1 / \pi \leq A \leq 1$, since the left hand side equals zero. The remaining case $1 \leq A \leq 2$ can be checked by hand.

\subsection{Proof of Lemma \ref{lem:asympt}}

We assume $I = (0,l)$, substitute $t = s \sqrt{c / \La}$ and write 
$$
\int_0^l \lk \La - \frac{c}{\delta(t)^2} \rk_+^{\sigma+1/2} dt \, = \, 2 \, \sqrt c \, \La^{\sigma} \int_1^{l\sqrt{\La}/(2\sqrt c)} \lk 1-\frac{1}{s^2} \rk^{\sigma+1/2} ds \, .
$$
The claim of the Lemma follows, if we show that 
$$
2 \sqrt c L^{cl}_{\sigma,1}   \int_1^{l\sqrt{\La}/(2\sqrt c)} \lk 1-\frac{1}{s^2} \rk^{\sigma+1/2} ds  - \sum_k \lk 1 - \frac{\pi^2 k^2}{\La \, l^2} \rk_+^\sigma  \, = \, \frac 12 - \sqrt c  + o \lk 1 \rk
$$
as $\La \to \infty$. With  $A = l \sqrt \La / \pi$ this is equivalent to
\begin{equation}
\label{in:lim}
2 \sqrt c \, L^{cl}_{\sigma,1}  \int_1^{\pi A /(2\sqrt{c})} \lk 1-\frac{1}{s^2} \rk^{\sigma+1/2} ds - \sum_k \lk 1- \frac{k^2}{A^2} \rk_+^\sigma  \, =  \,  \frac 12 - \sqrt c  + o \lk 1 \rk 
\end{equation}
as $A \to \infty$.

It is easy to see that 
\begin{equation*}
\sum_k \lk 1- \frac{k^2}{A^2}\rk_+^\sigma \, = \, \frac A2 B \lk \sigma+1,\frac 12 \rk - \frac 12 + o(1)
\end{equation*}
as $A \to \infty$. Moreover, we claim
\begin{equation}
\label{eq:limint}
\int_1^{\pi A /(2 \sqrt{c})} \lk 1 - \frac{1}{s^2} \rk^{\sigma+1/2} \, ds \, = \, \frac{\pi A}{2 \sqrt c} + \frac{1}{2} B\lk - \frac 12, \sigma + \frac 32 \rk + o(1)
\end{equation}
as $A \to \infty$ and (\ref{in:lim}) follows from $L^{cl}_{\sigma,1}  = B(\sigma+1,1/2)/(2 \pi)$, if we can establish (\ref{eq:limint}). Let us write
$$
\lk 1 - \frac{1}{s^2} \rk^m \, = \, \sum_{k\geq 0} (-1)^k  \binom{m}{k}  s^{-2k} \, = \, \sum_{k\geq0} \binom{k-m-1}{k} s^{-2k}
$$
for $m \geq 1$ and note that the sum is finite if $m \in \mathbb{N}$, while the sum converges uniformly on $s \in [1,\infty)$ if $m \notin \mathbb{N}$. Hence we have
$$
\int_1^y \lk 1- \frac{1}{s^2} \rk^m ds \, = \, y + \sum_{k\geq 0} \binom{k-m-1}{k}  \frac{1}{2k-1} + o(1)
$$
as $y \to \infty$. Using that
$$
\sum_{k\geq0} \binom{k-m-1}{k} \frac{1}{2k-1} \, = \,  \frac 12 B\lk - \frac 12, m+1 \rk 
$$
we obtain 
$$
\int_1^y \lk 1- \frac{1}{s^2}\rk^m ds \, = \, y + \frac 12 B \lk - \frac 12 , m+1 \rk + o(1)
$$
as $y \to \infty$, which is equivalent to (\ref{eq:limint}).
This finishes the proof of Lemma \ref{lem:asympt}.

\end{document}